\documentclass[reqno]{amsart}
\usepackage{amsmath,amsthm,amssymb}
\usepackage{latexsym}
\usepackage{eucal}

\newtheorem{theorem}{Theorem}[section]
\newtheorem{lemma}{Lemma}[section]

\theoremstyle{definition}

\newtheorem{proposition}{Proposition}

\def\cal{\mathcal}

\begin{document}
\title[Double exponential growth of the vorticity gradient\ldots]
{Double exponential growth of the vorticity gradient for the
two-dimensional Euler equation}

\author{Sergey A. Denisov}

\email{denissov@math.wisc.edu}

\thanks{{\it Keywords: Two-dimensional Euler equation, growth of the vorticity gradient} \\
\indent{\it 2000 AMS Subject classification:} primary 76B99,
secondary 76F99}

\address{University of Wisconsin-Madison,
Mathematics Department, 480 Lincoln Dr. Madison, WI 53706-1388, USA}

\maketitle

\maketitle
\begin{abstract}
For two-dimensional Euler equation on the torus, we prove that the
$L^\infty$--norm of the vorticity gradient can grow as double
exponential over arbitrary long but finite time provided that at
time zero it is already sufficiently large. The method is based on
the perturbative analysis around the singular stationary solution
studied by Bahouri and Chemin in \cite{chemin1}. Our result on the
growth of the vorticity gradient is equivalent to the statement that
the operator of Euler evolution is linearly unbounded in Lipschitz
norm for any time $t>0$.
\end{abstract} \vspace{1cm}

\section{Introduction and some upper bounds}

Consider the two-dimensional Euler equation for the vorticity
\begin{equation}
\dot\theta=\nabla\theta \cdot \psi, \quad \psi=\nabla^\perp
\Delta^{-1}\theta, \quad \theta(x,y,0)=\theta_0(x,y) \label{euler}
\end{equation}
and $\theta$ is $2\pi$--periodic in both $x$ and $y$ (that is, the
equation is considered on the torus $\mathbb{T}^2$). We assume that
$\theta_0$ has zero average over $\mathbb{T}^2$ and then
$\Delta^{-1}$ is well-defined since the Euler flow is
area-preserving and the average of $\theta(\cdot,t)$ is zero as
well. Denote the operator of Euler evolution by $\cal{E}_t$, i.e.
\[
\theta(t)=\cal{E}_t\theta_0
\]

 The global existence of the smooth solution for smooth
initial data is well-known and is due to Wolibner \cite{wol} (see
also \cite{March}). The estimate on the possible growth of the
Sobolev norms, however, is double exponential.  We sketch the proof
of this bound for $H^2$--norm. The estimates for $H^s, s>2$ can be
obtained similarly. More general results on regularity can be found
in \cite{const}. Let
\[
j_k(t)=\|\theta(t)\|_{H^k}
\]

\begin{lemma}If $\theta$ is the smooth solution of (\ref{euler}),
then
\begin{equation}
j_2(t)\leq \exp\Bigl(\frac{(1+2\log^+
j_2(0))\exp(C\|\theta_0\|_\infty t)-1}{2}\Bigr)\label{upper1}
\end{equation}

\end{lemma}
\begin{proof}
Acting on (\ref{euler}) with Laplacian we get
\[
\Delta \dot\theta=\Delta\theta_x \psi_y+2\nabla\theta_x\cdot \nabla
\psi_y-\Delta\theta_y\psi_x-2\nabla\theta_y\cdot \nabla \psi_x
\]
Multiply by $\Delta \theta$ and integrate over $\mathbb{T}^2$ to get
\begin{equation}
\partial_t\|\theta\|_{H^2}^2\lesssim  \|H(\psi)\|_\infty
\|\theta\|_{H^2}^2 \label{est1}
\end{equation}
where $H(\psi)$ denotes the Hessian of $\psi$. The next inequality
follows from the Littlewood-Payley decomposition (see \cite{const},
proposition~1.4 for more general result)
\begin{equation}
\|H(\psi)\|_\infty<C(\sigma) \|\theta\|_\infty(1+\log^+
\|\theta\|_{H^\sigma})\label{est2}
\end{equation}
for any $\sigma>1$. Notice that $\|\theta\|_\infty$ is invariant
under the motion so combine (\ref{est1}) and (\ref{est2}) to get
(\ref{upper1}).
\end{proof}
{\bf Remark 1.} In the same way one can prove bounds for higher
Sobolev norms, e.g.,
\begin{equation} \log j_4(t)\lesssim (1+\log^+
j_4(0))\exp(C\|\theta_0\|_\infty t)-1\label{upper2}
\end{equation}
The natural questions one can ask then are the following: first, how
fast can the Sobolev norms grow in time and what is the mechanism
that leads to their growth? Secondly, for fixed $t$, how does $\|
\cal E (t) \theta_0\|_{H^s}$ depend on $\|\theta_0\|_{H^s}$ when the
last expression grows to infinity? For example, given
$\|\theta_0\|_\infty \sim 1$, the right hand side in (\ref{upper1})
grows as a power function in $j_2(0)$, the degree grows
exponentially in $t$ and is more than one for any $t>0$.

Instead of working with Sobolev norms, we will be studying the
uniform norm of the vorticity gradient (or Lipschitz norm) as this
norm is more natural for the method used in the proof. It allows the
similar upper bound. We again give the sketch of the proof for
completeness.
\begin{lemma}
If $\theta_0$ is smooth and $\|\theta_0\|_\infty\sim 1$, then
\begin{equation}\label{upper31}
\|\nabla\cal E_t\theta_0\|_\infty \lesssim \exp \left( C(1+\log^+
\|\nabla\theta_0\|_\infty)e^{Ct}\right)
\end{equation}
\end{lemma}
\begin{proof}
If $\Psi(z,t)$ is area-preserving Euler diffeomorphism, then
\[
(\cal E_t\theta_0) (z)=\theta_0(\Psi^{-1}(z,t))
\]
On the other hand, $\Psi(z,t)$ solves
\[
\dot \Psi=-u(\Psi,t),\quad \Psi(z,0)=z
\]
where $u(z,t)=\nabla^\perp \Delta^{-1}\theta(\cdot, t)$. For the
Riesz transform we have a trivial estimate
\begin{equation}
\|H(\Delta^{-1}\theta)\|_\infty\lesssim 1+\|\theta\|_\infty
(1+\log^+ \|\nabla \theta\|_\infty)\label{ha}
\end{equation}
$\Bigl($ Indeed, without loss of generality we can evaluate the
integral at zero and assume $\theta(0)=0$. Then, e.g.,
\[
\left|\int_{B_1(0)} \frac{\xi_1 \xi_2}{|\xi|^4}
\theta(\xi)d\xi\right|\leq \int_{B_\delta(0)}\frac{1}{|\xi|^2}
|\theta(\xi)|d\xi+\int_{\delta<|z|<1}\frac{1}{|\xi|^2}
|\theta(\xi)|d\xi
\]
where $\delta^{-1}=\max\{ \|\nabla \theta\|_\infty, 2\}$. Apply now
the Lagrange formula to the first term to get (\ref{ha}). $\Bigr)$

so
\[
|u(w_1,t)-u(w_2,t)|\lesssim |w_1-w_2|b, \quad b=1+\log^+ \|\nabla
\theta(t)\|_\infty
\]
Therefore, we have
\[
|\dot f|\lesssim f b,\quad f(t)=|\Psi(z_2,t)-\Psi(z_1,t)|^2, \,
\]
After integration
\[
|z_2-z_1|\exp\left(-C\int_0^t b(\tau)d\tau\right)\leq
|\Psi(z_2,t)-\Psi(z_1,t)|\leq |z_2-z_1|\exp\left(C\int_0^t
b(\tau)d\tau\right)
\]
Since
\[
\|\nabla
\theta(z,t)\|_\infty=\sup_{z_1,z_2}\frac{|\theta_0(\Psi^{-1}(z_2,t))-
\theta_0(\Psi^{-1}(z_1,t))|} {|z_2-z_1|}
\]
we get inequality
\[
\|\nabla \theta(z,t)\|_\infty\lesssim \|\nabla \theta_0\|_\infty\exp
\left(C\int_0^t b(\tau)d\tau\right)
\]
   Taking $\log$ of the both parts and applying the
Gronwall-Bellman, we get (\ref{upper31}).

\end{proof}

In this paper, we will work only with large $\|\nabla
\theta_0\|_\infty$. For that case, we will show that, given
arbitrarily large $\lambda$, the estimate $\max_{t\in
[0,T]}\|\nabla\theta(\cdot,t)\|_\infty>\lambda^{e^{T}-1}\|\nabla\theta_0\|_\infty$
can hold for some infinitely smooth initial data. This is far from
showing that (\ref{upper1}) or (\ref{upper31}) are sharp however it
already is equivalent to the statement that $\cal E_t$ is linearly
unbounded. The question of whether $\|\nabla\theta_0\|_\infty$ can
be taken $\sim 1$ is left wide open, see discussion in the last
section.

Our results rigorously confirm the following observation: if the 2D
incompressible inviscid fluid dynamics gets into a certain
``instability mode'' then the Sobolev norms can grow very fast in
local time (i.e. counting from the time the ``instability regime"
was reached). Can the Sobolev norms grow at all infinitely in time
assuming that initially they are small? The answer to this question
is yes, see \cite{den} and \cite{KN,Nad,Yud,Yud1, shnir}. The
important questions of linear and nonlinear instabilities were
addressed before (see, e.g., \cite{fv} and references there). In the
recent paper \cite{hm}, it was proved that $\cal{E}_t$ is not
uniformly continuous on the unit ball in Sobolev spaces.

{\bf Remark 2. }It must be mentioned here that 2D Euler allows
rescaling which provides the tradeoff between the size of $\theta_0$
and the speed of the process, i.e. if $\theta(x,y,t)$ is a solution
then $\mu\theta(x,y,\mu t)$ is also a solution for any $\mu>0$.
However, in our construction we will always have $\|\theta\|_p\sim
1, \quad \forall p\in [1,\infty]$.

{\bf Remark 3.} If one replaces $\Delta^{-1}$ in (\ref{euler}) by
$\Delta^{-\alpha}$ with $\alpha>1$, then the growth of the vorticity
gradient is at most exponential, e.g.
\[
\|\nabla \cal E^{(\alpha)}_t \theta_0\|_\infty\lesssim \|\nabla
\theta_0\|_\infty \exp(C\|\theta_0\|_\infty t)
\]
Moreover, the lower exponential bound can hold for all times as long
as $\theta_0$ is properly chosen (see \cite{den}).

\section{The singular stationary solution and  dynamics on the
torus}

The following singular stationary solutions was studied before (see,
e.g., \cite{chemin1, chemin2} in the context of $\mathbb{R}^2$). We
consider the following function
\[
\theta^s_0(x,y)=sgn(x)\cdot sgn(y), \quad |x|\leq \pi, |y|\leq \pi
\]
This is a steady state. Indeed, the function
$\psi_0=\Delta^{-1}\theta_0^s$ is odd with respect to each variable
as can be verified on the Fourier side. That, in particular, implies
that $\psi_0$ is zero on the coordinate axes so its gradient is
orthogonal to them. This steady state, of course, is a weak
solution, a vortex-patch steady state. Another consequence of
$\psi_0$ being odd is that the origin is a stationary point of the
dynamics.

By the Poisson summation formula, we have
\[
\sum_{n\in \mathbb{Z}^2,n\neq (0,0)} |n|^{-2} e^{in\cdot z} =C\ln
|z|+\phi(z), \quad z\sim 0
\]
where $\phi(z)$ is smooth and even.

 Therefore,
around the origin we have
\[
\nabla \psi_0(x,y)\sim \int  \int_{B_{0.5}(0)}
\frac{(x-\xi_1,y-\xi_2)}{(x-\xi_1)^2+(y-\xi_2)^2}sgn(\xi_1)
sgn(\xi_2)d\xi_1d\xi_2+(O(y),O(x))
\]
Due to symmetry, it is sufficient to consider the domain
$D=\{0<x<y<0,001\}$. Then, taking the integrals, we see that
\begin{equation}
\mu(x,y)=(\mu_1,\mu_2)= \left(\nabla^\perp \psi_0\right)(x,y)=
\label{potok}
\end{equation}
\begin{eqnarray*}
=c_1\left(-\int_0^x \ln(y^2+\xi^2)d\xi+xr_1(x,y),\int_0^y \ln(x^2+\xi^2)d\xi +yr_2(x,y)\right) \\
=c_2(-x\log y +xO(1), y\log y+yO(1)) \quad {\rm if} \,
 (x,y)\in D
\end{eqnarray*}
The correction terms $r_{1(2)}$ are smooth.
 Without loss
of generality we will later assume that $c_2=1$ in the last formula
(so $c_1=0.5$). That can always be achieved by time-rescaling.
Notice also that the flow given by the vector-field $\mu$ is
area-preserving.

Thus, the dynamics of the point $(\alpha,\beta)\in D, \alpha\ll
\beta$ is
\begin{equation}
(C_1\beta)^{e^t}\lesssim y(t)\lesssim  (C_2\beta)^{e^t}, \quad
\alpha(C_1\beta)^{-e^t+1}\lesssim x(t)\lesssim
\alpha(C_2\beta)^{-e^t+1},\, t\in [0,t_0],\,
 \label{dex}
\end{equation}
where $t_0$ is the time the trajectory leaves the domain $D$. These estimates therefore give a bound on $t_0$.
 The attraction to the origin, the stationary point, is
double exponential along the vertical axis and the repulsion along
the horizontal axis is also double exponential.

\section{The idea}

The idea of constructing the smooth initial data for a double
exponential scenario is quite simple and roughly can be summarized
as follows: given any $T>0$, we will smooth out the singular steady state such that
the dynamics is double exponential over $[0,T]$ in a certain domain  away from the
coordinate axes. Then we will place a small but steep bump in the
area of double exponential behavior and will let it evolve hoping
that the vector field generated by this bump itself is not going to
ruin the double exponential contraction in $OY$ direction. The rest
of the paper verifies that this indeed is the case.

\section{The Model Equation}

Consider the following system of ODE's
\begin{equation}
\left\{\begin{array}{cc}
\dot{x}=\mu_1(x,y)+\nu_1(x,y,t),\quad x(\alpha,\beta,0)=\alpha\\
\dot{y}=\mu_2(x,y)+\nu_2(x,y,t), \quad y(\alpha,\beta,0)=\beta
\end{array}\right.\label{model}
\end{equation}
Here we assume the following
\begin{equation}
|\nu_{1(2)}|<0,0001\upsilon r,\quad r=\sqrt{x^2+y^2} \label{usl1}
\end{equation}
and
\begin{equation}
 \, |\nabla\nu_{1(2)}|<0,0001\upsilon \label{usl2}
\end{equation}
with small $\upsilon$ (to be specified later) and these estimates
are valid in the area of interest
\[
\aleph=\{y>\sqrt x\}\cap \{y<\epsilon_2\}\cap \{x>\epsilon_1\}
\]
where
\[
\upsilon\ll \epsilon_1\ll \epsilon_2
\]
The functions $\nu_{1(2)}$ are infinitely smooth in all variables in
$\aleph$ but we have no control over higher derivatives. We also
assume that the flow given by (\ref{model}) is area preserving. Our
goal is to study the behavior of trajectories within the time
interval $[0,T]$. In this section, the parameters will eventually be
chosen in the following order
\[
T\longrightarrow \epsilon_2\longrightarrow \epsilon_1\longrightarrow
\upsilon
\]
Here are some obvious observations:

1. If $\epsilon_{1(2)}$ are small and
\begin{equation}
\alpha \gtrsim \upsilon \left|\frac{\beta}{\log \beta}\right| \label{odno}
\end{equation}
then $x(t)$ increases and $y(t)$ decreases. This monotonicity
persists as long as the trajectory stays within $\aleph$. Assuming
that $\epsilon_{1(2)}$ are fixed, (\ref{odno}) can always be
satisfied by taking $\upsilon$ small enough, i.e.,

\begin{equation}
\upsilon\lesssim \frac{\epsilon_1|\log\epsilon_2|}{\epsilon_2}\label{uslovie1}
\end{equation}

2. We have estimates
\begin{equation}
x(-\log y+C)+\upsilon y>\dot{x}>x(-\log y-C)-\upsilon y, \quad -y(|\log
y|+C)<\dot{y}<-y(|\log y|-C)\label{rat1}
\end{equation}
The second estimate yields
\begin{equation}
e^{e^t(\log\beta+C)}> y(t)>
 e^{
e^t(\log\beta-C)} \label{rat2}
\end{equation}
Let us introduce
\[
\kappa(T,\beta)=e^{ e^T(\log\beta-C)}
\]
For $x(t)$, we have
\[
x(t)\leq\alpha\exp\left(Ct-\int_0^t \log
y(\tau)d\tau\right)+\upsilon\int_0^ty(\tau)\exp\left(C(t-\tau)-\int_\tau^t
\log y(s)ds\right)d\tau
\]
\[
x(T)<(\alpha+\upsilon\beta T) \exp\Bigl(T(C+|\log
\kappa(T,\beta)|)\Bigr)
\]
Thus, the trajectory will stay inside $\aleph$ for any $t\in [0,T]$ as
long as
\[
\alpha<\kappa^{3+T}-\upsilon\epsilon_2T
\]
and if we have
\begin{equation}
\upsilon<\kappa^{4+T}(T,\beta) \label{uusl}
\end{equation}
then the condition
\begin{equation}
\alpha<\beta^{8e^{2T}}\label{dva}
\end{equation}
is sufficient for the trajectory to stay inside $\aleph$ for $t\in [0,T]$. Thus, we are
taking
\[
\epsilon_1<\epsilon_2^{8e^{2T}}
\]
and focus on the nonempty domain
\[
\Omega_0=\{(\alpha,\beta): \epsilon_1<\alpha<\beta^{8e^{2T}},
\beta<\epsilon_2\}
\]
The condition on $\upsilon$  is (\ref{uusl}), so taking the smallest
possible $\kappa(T,\beta)$ within $\Omega_0$ we get, e.g.,
\begin{equation}
\upsilon<\epsilon_1^{10} \label{uslovie2}
\end{equation}
Then, any point from $\Omega_0$  stays inside $\aleph$ over $[0,T]$,
$x(t)$ grows monotonically and $y(t)$ monotonically  decays with the
double-exponential rate given in (\ref{rat2}).

Now, we will prove that the derivative in $\alpha$ of
$x(\alpha,\beta,t)$ grows with the double-exponential rate and this
will be the key calculation. For any $t\in [0,T]$, (\ref{potok})
yields
\begin{equation}
\left\{
\begin{array}{lc}
\displaystyle \dot{x}_\alpha=-0.5 x_\alpha\log (x^2+y^2)+x_\alpha r_1+\\
\hspace{2cm}+xx_\alpha r_{1x}+xy_\alpha r_{1y}+\nu_{1x}
x_\alpha+\nu_{1y}y_\alpha-y_\alpha
\arctan(xy^{-1}) \\
\dot{y}_\alpha=0.5y_\alpha\log (x^2+y^2)+y_\alpha r_2+yx_\alpha
r_{2x}+\\
\hspace{2cm}+yy_\alpha
r_{2y}+\nu_{2x}x_\alpha+\nu_{2y}y_\alpha+x_\alpha \arctan(yx^{-1}) &
\end{array}
\right.
\end{equation}
and $x_\alpha(\alpha,\beta,0)=1$, $y_\alpha(\alpha,\beta,0)=0$. Let
\[
f_{11}(t)=\nu_{1x}-0.5 \log (x^2+y^2)+r_1+xr_{1x}
\]
\[
 f_{12}(t)=xr_{1y} +\nu_{1y}-\arctan(xy^{-1})
\]
\[
f_{21}(t)=yr_{2x}+\nu_{2x}+\arctan(yx^{-1})
\]
\[
 \,f_{22}(t)=0.5\log (x^2+y^2)+r_2+yr_{2y}+\nu_{2y}
\]

\[
x_\alpha=\exp\left(\int_0^t f_{11}(\tau)d\tau\right)\hat x, \quad
y_\alpha=\exp\left( \int_0^t f_{22}(\tau)d\tau\right)\hat y
\]
If
\[
g=f_{11}-f_{22}
\]
then
\[
\hat x(t)=1+\int_0^t \hat x(s) f_{21}(s) \int_s^t
f_{12}(\tau)\exp\left(-\int_s^\tau g(\xi)d\xi\right)d\tau ds
\]
Since the trajectory is inside $\aleph$, we have $y>\sqrt x$ and so
\[
|f_{12}|\lesssim y+\upsilon, \quad |f_{21}|\lesssim 1 ,\quad
f_{11}>e^t(-\log \beta+C), \quad g(t)>1
\]
From (\ref{rat2}),  we get
\[
|\hat x(t)-1|\lesssim\upsilon\int_0^t |\hat x (\tau)|d\tau+\int_0^t
|\hat x(s)| \left(\int_s^t e^{e^\tau (\log\beta+C)} e^{-(\tau-s)}
d\tau\right) ds
\]
The following estimate is obvious
\[
\int_s^t e^{e^\tau (\log\beta+C)} e^{-(\tau-s)}  d\tau \ll e^{-s}
\]
as $\beta$ is small.
Assuming that
\begin{equation}
\upsilon\ll (T+1)^{-1}\label{uslovie3}
\end{equation}
and $\epsilon_2$ is small,
 we have
\[
\hat x(t)\sim 1
\]
and
\begin{equation}
x_\alpha(\alpha,\beta,T)>\left(\frac 1\beta\right)^{(e^{T}-1)/2}
\label{klyuch}
\end{equation}
The estimate (\ref{klyuch}) is the key estimate that will guarantee the necessary growth.

Now, let us place a circle $S_\gamma(\tilde{x},\tilde{y})$ with
radius $\gamma$ and center at $(\tilde{x},\tilde{y})$ into the zone
$\Omega_0$. Consider also the line segment $l=[A_1,A_2]$,
$A_1=(\tilde{x}-\gamma/2,\tilde{y}),
A_2=(\tilde{x}+\gamma/2,\tilde{y})$ in the center, parallel to $OX$.
We will track the evolution of this disc and this line segment under
the flow. We have by the Lagrange formula

\[
x(A_2,T)-x(A_1,T)>\beta^{-(e^{T}-1)/2}|A_2-A_1|
\]
From the positivity of $x_\alpha(\alpha,\beta,T)$ it follows that
the image of $l$ under the flow is a curve given by the graph of a
smooth function $\Gamma(x)$. Thus, the image of $l$ (call it $l'$)
has length at least $\beta^{-(e^{T}-1)/2}|A_2-A_1|$. Denote the
distance from $l'$ to $S'_\gamma(\tilde{x},\tilde{y})$, the image of
the circle, by $d$. Then, the domain
$\{\Gamma(x)-d<y(x)<\Gamma(x)+d, x\in (x(A_1,T),x(A_2,T))\}$ is
inside $B'_\gamma(\tilde{x},\tilde{y})$. The area of this domain is
at least
\[
d\cdot \beta^{-(e^{T}-1)/2}|A_2-A_1|
\]
Thus, assuming that the flow preserves the area, we have
\[
d\lesssim \beta^{(e^{T}-1)/2}\gamma
\] Consequently, if we place a bump in $\Omega_0$ such that the $l$ and
 $S_\gamma(\tilde{x},\tilde{y})$ correspond to  level sets, say,
$h_2$ and $h_1$ ({\bf and, what is crucial, $h_{1(2)}$ are
essentially arbitrary $0<h_1<h_2<0,0001$}), then the original slope
of at least $\sim|h_2-h_1|/\gamma$ will become not less than
\begin{equation}
 \beta^{-(e^{T}-1)/2}\cdot\left(|h_2-h_1|/\gamma\right) \label{eex}
\end{equation}
thus leading to double-exponential growth of arbitrarily large
gradients.

{\bf Remark 1.} If $\beta$ is a fixed small number, we have growth
in $T$. If $T$ is any positive fixed moment of time, we have the
growth if $\beta\to 0$.

{\bf Remark 2.} Let us reiterate the order in which the parameters
are chosen: we first fix any $T$, then small $\epsilon_2$, then
$\epsilon_1<\epsilon_2^{8e^{2T}}$. How small $\epsilon_2$ must be
taken will be determined by how large the parameter $\lambda$ is
chosen in the theorem \ref{main} below. This defines the set
$\Omega_0$. For the whole argument  to work we need to collect all
conditions on $\upsilon$: (\ref{uslovie1}), (\ref{uslovie2}),
(\ref{uslovie3}) which leads to
\begin{equation}
\upsilon<\epsilon_1^{10} \label{choiceups}
\end{equation}

\section{Small perturbations of a singular cross can also generate  double
exponential contraction in $\aleph$}

Assume that the function $\theta_1$ at any given time $t\in [0,T]$
is such that
\[
\theta_1(x,y,t)=\theta_0^s(x,y)
\]
outside the ``cross"-domain $A=\{|x-\pi k|<\tau\}\cup \{|y-\pi
l|<\tau\}$ where $\tau$ is small and $k,l\in \mathbb{Z}$.  Inside
the domain $A$ we only assume that $\theta_1$ is bounded by one in
absolute value, is even, and has zero average. Notice here that the
Euler flow preserves property of the function to be even. Given
fixed $\epsilon_{1(2)}$ and the domain $\aleph$ defined by these
constants, we are going to show that the flow generated by
$\theta_1$ can be represented in $\aleph$ in the form (\ref{model})
with $\upsilon(\tau)\to 0$ as $\tau\to 0$. We assume of course that
$\tau\ll \epsilon_1$.

For that, we only need to study
\[
F_1=\nabla\Delta^{-1}p, \quad p=\theta_1-\theta_0^s
\]
Here are some obvious properties of $F_1$

1. $F_1(0)=0$ as $\theta_{1}$ and $\theta_0^s$ are both even.

2. We have
\[
F_1(z)\sim \int_A \left(
\frac{\xi-z}{|\xi-z|^2}-\frac{\xi}{|\xi|^2}\right)p(\xi)d\xi
\]
Using the formula
\[
\left| \frac{x}{|x|^2}-\frac{y}{|y|^2}\right|=\frac{|x-y|}{|x|\cdot
|y|}
\]
we get
\[
|F_1(z)|\lesssim |z|\frac{\tau|\log\tau|}{\epsilon_1}
\]
Thus, by taking $\tau$ small, we can satisfy $(\ref{usl1})$. How
about (\ref{usl2})? For the Hessian, we have
\[
|H\Delta^{-1}p|\lesssim \epsilon_1^{-2}\tau
\]
and after combining we must have
\begin{equation}
 \epsilon_1^{-2}\tau|\log\tau|\lesssim \epsilon_1^{10} \label{sizecross}
\end{equation}
by (\ref{choiceups}). Thus, this condition on the size of the cross
guarantees that the arguments in the previous section work.

\section{The flow generated by a small steep bump in $\aleph$}

In this section, we assume that at a given moment $t\in [0,T]$, we
have a smooth even function $b(x,y,t)$ with support in $\aleph\cup
-\aleph$, with zero average, and
\[
\|b\|_2<\omega, \quad \|\nabla b\|_\infty <M
\]
(here one should think about small $\omega$ and large $M$).
We will study the flow generated by this function. Let
\[
F_2=\nabla \Delta^{-1} b
\]
Here are some properties of $F_2$

1. $F_2(0)=0$.

2. To estimate the Hessian of $\Delta^{-1}b$, consider the second
order derivatives. For example,
\[
(\Delta^{-1}b)_{\alpha\beta}(\alpha,\beta)\sim
\int\limits_{(\alpha-\xi)^2+(\beta-\eta)^2<1}
\frac{(\alpha-\xi)(\beta-\eta)}{((\alpha-\xi)^2+(\beta-\eta)^2)^2}b(\xi,\eta,t)d\xi
d\eta=
\]
\[
=\int_{1>(\alpha-\xi)^2+(\beta-\eta)^2>\rho^2}
\frac{(\alpha-\xi)(\beta-\eta)}{((\alpha-\xi)^2+(\beta-\eta)^2)^2}b(\xi,\eta,t)d\xi
d\eta\quad+
\]
\[
\int\limits_{(\alpha-\xi)^2+(\beta-\eta)^2<\rho^2}
\frac{(\alpha-\xi)(\beta-\eta)}{((\alpha-\xi)^2+(\beta-\eta)^2)^2}
\left[b(\alpha,\beta,t)+\nabla b(\xi',\eta',t)\cdot
(\xi-\alpha,\eta-\beta)\right] d\xi d\eta
\]
The first term is controlled by  $\omega \rho^{-1}$.
 By our assumption, the second term is dominated
by $M\rho$. Optimizing in $\rho$ we have
\[
\|H\Delta^{-1}b\|_\infty\lesssim \sqrt{M\omega}
\]
To guarantee the  conditions that lead to double exponential growth
with arbitrary a priory given $M$, we want to make $\omega$ so small
that conditions (\ref{usl1}) and (\ref{usl2}) are satisfied with
$\upsilon$ as small as we need (i.e., (\ref{choiceups})). The
condition $(\ref{usl2})$ is immediate and $(\ref{usl1})$ follows
from $F_2(0)=0$, Lagrange formula and the estimate on the Hessian.

\section{One stability result and the proof of the main theorem}

It is well known that given $\theta_0\in L^\infty(\mathbb{T}^2)$,
the weak solution exists and the flow can be defined by the
homeomorphic maps $ \Psi_{\theta_0}(x,y,t) $ for all $t$ so that $
\theta(x,y,t)=\theta_0(\Psi^{-1}_{\theta_0}(x,y,t)) $ where
$\Psi_{\theta_0}$ itself depends on $\theta_0$. The continuity of
this map  though is rather poor  (\cite{chemin2}, theorem 2.3,
p.99). In this section, we will need to take smooth $\theta_0$ such
that  \[\max_{t\in [0,T]}\max_{z\in \mathbb{T}^2}
|\Psi_{\theta_0}(z,t)-\Psi_{\theta_0^s}(z,t)|\to 0\] To this end, we
will  consider  $\theta_0=\theta_0^s$ outside the domain $\cal{D}$
of small area. Inside this domain we assume $\theta_0$ to be bounded
by some universal constant. The proof of Yudovich theorem (see,
e.g., the argument on pp. 313--318, proof of Proposition 8.2,
\cite{mb}) implies

\begin{equation}\label{stability}
\max_{t\in [0,T]}\max_{z\in \mathbb{T}^2}
|\Psi_{\theta_0^s}(z,t)-\Psi_{\theta_0}(z,t)|\to 0
\end{equation}
as $|\cal{D}|\to 0$.

This is the only stability result with respect to initial data that we are going to need in the argument below.

\begin{theorem}\label{main}
 For any large $\lambda$ and any $T>0$, we can find smooth initial
 data $\theta_0$ so that $\|\theta_0\|_\infty<2$  and

\[
\max_{t\in
[0,T]}\|\nabla\theta(\cdot,t)\|_\infty>\lambda^{e^{T}-1}\|\nabla\theta_0\|_\infty
\]
\end{theorem}
\begin{proof}

 Fix any $T>0$ and find $\epsilon_{1(2)}$. For larger $\lambda$,
 we have to take smaller $\epsilon_2$ (see remark 1 in the fourth section).  Identify the domain $\Omega_0$
 and place a bump (call it $b(z)$) in $\Omega_0\cup
-\Omega_0$ so that the resulting function is even. Make sure that
this bump has zero average, height $h_2$ and diameter of support
$h_1$ so that the gradient initially is of the size $\sim h_2/h_1$.
Here $h_1\ll h_2\ll 1$ will be adjusted later.

Take a smooth even function $\omega(x,y)$ supported on $B_1(0)$
such that
\[
\int_{\mathbb{T}^2} \omega(x,y)dxdy=1
\]
For positive small $\sigma$, consider
\[
\tilde\theta_\sigma(x,y)=\theta_0^s\ast \omega_\sigma \in C^\infty,
\quad \omega_\sigma=\sigma^{-2}\omega(x/\sigma,y/\sigma)
\]
We take $\sigma\ll \epsilon_1$ so  $\tilde\theta_\sigma(x,y)$ and
$\theta_0^s(x,y)$ coincide in $\aleph$.

As the initial data for Euler dynamics we take a sum
\[
\tilde\theta_\sigma(z)+b(z)
\]
Then, since $\theta_0^s$ is stationary under the flow, the stability
result (\ref{stability}) guarantees that given any $\tau$ and
keeping the same value of $h_2/h_1$, we can find $\sigma$ and $h_1$
so small that over the time interval $[0,T]$ we satisfy

1. The ``evolved bump'' $b(z,t)$ stays in the domain $\aleph$ (e.g.,
$\Psi_{\theta_0}(t)\Bigl( {\rm supp }\,b(z) \Bigr) \subset \aleph$).

2. Outside the cross of size $\tau$ (the one considered in section
five) and the support of the evolved bump $b$, the solution is
identical to $\theta_0^s$.

  Fix  $\sigma$ and $h_1'$ so small that for any $h_1<h_1'$ we have
 the size of $A$ being small, i.e. $\tau$ is small as we wish. The value of
$\tau$ must be small enough to ensure the double exponential
scenario, the conditions (\ref{usl1}) and (\ref{usl2}). For that, we
need (\ref{sizecross}).

Next, we proceed by contradiction. Assume that for all $t\in [0,T]$
we have $\|\nabla\theta(z,t)\|_\infty<M=(h_2/h_1) \lambda^{e^T-1}$.
Then, because $\|b(z,t)\|_2$ is constant in time as the flow is
area-preserving and $\|b(z,t)\|_2\lesssim h_1h_2$, we only need to
take $h_2$ so small that $ \sqrt{Mh_1h_2} $ is small enough to
guarantee the double exponential scenario and the estimate
(\ref{eex}). This gives us a contradiction as the double exponential
scenario makes the gradient's norm more than $M$ (provided that
$\epsilon_2\ll \lambda^{-2}$). For the initial value,
\[
\|\nabla \theta_0\|_\infty\sim \sigma^{-1}+h_2/h_1\sim h_2/h_1
\]
by arranging $h_{1(2)}$ (and keeping $h_1<h_1'$).

Here is an order in which parameters are chosen in  this
construction:
\[
\{T,\lambda\}\longrightarrow\epsilon_2\longrightarrow\epsilon_1\longrightarrow
\{\sigma,h_{1(2)}\}
\]
\end{proof}
\section{The operator $\cal{E}_t$ is linearly unbounded.}

The theorem \ref{main} is equivalent to the following
\begin{proposition}
The operator $\cal E_t$ is linearly unbounded for any $t>0$, i.e.
\[
\sup_{\theta_0\in C^\infty(\mathbb{T}^2), 0<\|\theta_0\|_\infty \leq
1, \theta_0\perp 1}\frac{\|\nabla \cal E_t
\theta_0\|_\infty}{\|\nabla \theta_0\|_\infty}=+\infty
\]
\end{proposition}
\begin{proof}
The proof is immediate. Indeed, given any fixed $t$, we have
\[
\sup_{\tau\in [0,t]}\left(\sup_{\theta_0\in C^\infty(\mathbb{T}^2),
\|\theta_0\|_\infty = 1, \theta_0\perp 1}\frac{\|\nabla \cal
E_{\tau} \theta_0\|_\infty}{\|\nabla
\theta_0\|_\infty}\right)=+\infty
\]
by taking $\lambda\to\infty$ in the theorem \ref{main}. Then, to
have the statement at time $t$, we only need to multiply $\theta_0$
by a suitable number and use remark 2 from the first section.

Vice versa, in the theorem \ref{main} the combination
$\lambda^{e^{T}-1}$ can be replaced by arbitrary large number. In
this formulation, the statement follows from the proposition.

\end{proof}

As the statement of the theorem \ref{main} holds with any $\lambda$,
the double exponential function is not relevant at all in the
formulation itself. However, it is very special hyperbolic scenario
with double exponential rate of contraction that ultimately provided
the superlinear dependence on the initial data.

The interesting and important question is whether the vorticity
gradient can grow in the same double exponential rate starting with
initial value $\sim 1$? We do not know the answer to this question
yet and the best known bound is (see, e.g.,  \cite{den})
\[
\max_{t\in [0,T]} \|\nabla\theta(\cdot, t)\|_\infty >e^{0.001 T}
\]
for arbitrary $T$ and for $T$--dependent $\theta_0$ with
$\|\theta_0\|_\infty\sim \|\nabla\theta_0\|_\infty\sim 1$.

\section{Acknowledgment}

This research was supported by NSF grants DMS-1067413 and
DMS-0758239. The hospitality of the Institute for Advanced Study,
Princeton, NJ is gratefully acknowledged. We are grateful to T. Tao
for pointing out that the theorem \ref{main} is equivalent to
$\cal{E}_t$ being linearly unbounded and to R. Killip and A. Kiselev
for interesting comments.

\end{document}